\documentclass{amsart}
\usepackage{graphicx}
\usepackage{graphicx}
\usepackage{amssymb,amscd,amsthm,amsxtra}
\usepackage{latexsym}
\usepackage{epsfig}

\newtheorem{thm}{Theorem}[section]
\newtheorem{cor}[thm]{Corollary}
\newtheorem{lem}[thm]{Lemma}
\newtheorem{prop}[thm]{Proposition}
\theoremstyle{definition}

\theoremstyle{remark}
\newtheorem{rem}[thm]{Remark}
\numberwithin{equation}{section}

\newcommand{\R}{\mathbb R}
\newcommand{\ov}{\overline}
\newcommand{\s}{\mathbb S}
\newcommand{\h}{\mathbb H}
\newcommand{\eps}{\varepsilon}

\newcommand{\be}{\begin{equation}}
\newcommand{\ee}{\end{equation}}
\begin{document}

\title[Radial graphs of constant mean curvature in Hyperbolic space]{Rearrangements and radial graphs of constant mean curvature in Hyperbolic space}%
\author{D. De Silva}
\address{Department of Mathematics, Barnard College, Columbia University, New York, NY 10027}
\email{\tt  desilva@math.columbia.edu}
\author{J. Spruck}%
\address{Department of Mathematics, Johns Hopkins University, Baltimore, MD 21218}%
\email{\tt  js@math.jhu.edu}%


\begin{abstract} We investigate the problem of finding smooth hypersurfaces
of constant mean curvature in hyperbolic space, which
can be represented as radial graphs over a subdomain of the upper hemisphere. Our approach is variational and our
main results are proved via rearrangement techniques.

\end{abstract}
\maketitle

\section{Introduction}

In this paper we study the problem of finding smooth hypersurfaces
of constant mean curvature in hyperbolic space $\h^{n+1}$, which
can be represented as radial graphs over a domain $\Omega$
strictly contained in the upper hemisphere $\s^n_+ \subset
\R^{n+1}$. This also leads by an approximation process
to the existence and uniqueness of smooth complete hypersurfaces
of constant mean curvature $H \in (-1,1)$ with prescribed asymptotic
boundary $\Gamma$ at infinity, in  case  $\Gamma $ is the boundary of a continuous star-shaped
domain.

We use the half-space model,

$$\h^{n+1}=\{(x,x_{n+1}) \in \R^{n+1} | x_{n+1}>0\}$$

\

\noindent equipped
with the hyperbolic metric $$ds_H^2 = \frac{1}{x_{n+1}^2}ds_E^2,$$

\

\noindent where $ds_E^2$ denotes the Euclidean metric on $\R^{n+1}.$

Let $\Omega \subset \mathbb{S}^n_+$, and  suppose that $\Sigma$ is
a radial graph over $\Omega$ with position vector $X$  in $\R^{n+1}$.
Then we can write

\be\label{graph}X=e^{v(z)} z, \quad z \in \Omega,\ee

\

\noindent for a function $v$
defined over $\Omega$. Assume that $\Sigma$ has constant mean curvature $H$ in hyperbolic
space with respect to the outward unit  normal.  Then $v$ satisfies the
divergence form elliptic equation

\begin{equation}\label{dirichlet1}\textrm{div}_{z}\left(\frac{y^{-n}\nabla{v}}{\sqrt{1+|\nabla v|^2}}\right)=
nH y^{-(n+1)} \quad \text{in} \ \Omega, \end{equation}

\

\noindent where $y=z_{n+1}$ and the divergence and gradient are with respect to the standard metric on the sphere.

We apply direct methods of the calculus of variations, in order to
prove the existence of a smooth solution to \eqref{dirichlet1}.
Let,

\be \label{functional}
\mathcal{I}_\Omega(v) := \int_{\Omega} \sqrt{1+|\nabla v|^2} ~y^{-n}dz  + nH\int_{\Omega}v(z)y^{-(n+1)}dz, \ee

\medskip

\noindent be the energy functional associated to equation \eqref{dirichlet1}. In this variational setting, we will easily obtain the existence
of bounded local minimizers of $\mathcal{I}_\Omega(\cdot)$ in the class
$BV(\Omega)$, as long as $|H|<1$.  However the Dirichlet problem in this
generality needs to be carefully formulated, see Section 2.

Our main objective is to prove a
regularity result which guarantees that such minimizers are
smooth, and hence the associated graphs \eqref{graph} are smooth hypersurfaces of constant mean curvature
in $\h^{n+1}$. We first prove the following result.
\begin{thm}\label{lowd} Assume $n \leq 6$ and let $v \in BV(\Omega) \cap L^{\infty}(\Omega)$ be a local minimizer to
$\mathcal{I}_\Omega(\cdot).$ Then $v \in C^{\infty}(\Omega).$
\end{thm}

The elegance of this low dimensional result lies in the fact that
it does not require any kind of a priori gradient bounds, which in
this context may appear computationally tedious.
The proof  is based on the connection between
non-parametric (radial graphs) and parametric  surfaces
of constant mean curvature in hyperbolic space. For the latter,
regularity in low dimensions is well-known (see for example
\cite{S}). We exploit this fact and recover the same regularity result
for radial graphs, via rearrangement techniques. A similar approach
 has been followed in the Euclidean setting to find smooth vertical graphs
 of prescribed mean curvature (see for example \cite{G}). See also \cite{D}
 for an existence and regularity result for a degenerate equation obtained
 via similar techniques.

In order to remove the low dimensional constraint we first analyze the case when the domain $\Omega$
satisfies an appropriate assumption. This allows us to set up and solve the Dirichlet problem for $\mathcal{I}_\Omega(\cdot)$ and obtain smoothness of the minimizer from the smoothness of the boundary data. Indeed, we prove the following result which requires the
construction of appropriate barriers.

\begin{thm}\label{existence} Let $\Omega$ be a subdomain of $\mathbb{S}^n_+$  with $\partial \Omega \in C^2$, and let $\gamma$ be a continuous radial graph over $\partial \Omega$. Let $h$ be  the hyperbolic mean curvature  of the radial cone over $\partial \Omega$ restricted to  $\partial \Omega$.  Then if $h> |H|$ , there exists a unique smooth radial graph
 $\Sigma$  of constant mean curvature H in $\h^{n+1}$ (defined over $\Omega$)  with  boundary $\gamma.$
 \end{thm}

Then using standard approximation techniques, a corollary of Theorem \ref{existence}, and an interior
gradient bound which is of independent interest, we prove the following
result.

\begin{thm}\label{anyd} Let $v \in BV(\Omega) \cap L^{\infty}(\Omega)$ be a local minimizer to
$\mathcal{I}_\Omega(\cdot).$ Then $v \in C^{\infty}(\Omega).$
\end{thm}

Finally, by a limiting argument using the afore mentioned barriers we
recover the following result from \cite{GS}.

\begin{thm}\label{global}Let $\Gamma$ be the boundary of a continuous star-shaped
domain in $\R^n$ and let $|H|<1$. Then there exists a
unique hypersurface $\Sigma$ of constant mean curvature $H$ in
$\h^{n+1}$ with asymptotic boundary $\Gamma.$ Moreover $\Sigma$
may be represented as the radial graph over $\s^n_+ \subset
\R^{n+1}$ of a function in $C^\infty(\s_+^n) \cap
C^0(\ov{\s_+^n}).$
\end{thm}

The paper is organized as follows. In Section 2, after briefly introducing
some notation, we set up and solve the Dirichlet problem for our energy
functional in the class of $BV$ functions. Then, in Section 3, we prove
our low dimensional result Theorem \ref{lowd}, via rearrangement techniques.
The proof of Theorem \ref{existence} is exhibited in Section 4. In Section 5,
we present the proof of an interior gradient bound for smooth solutions to equation
\eqref{dirichlet1} and then we apply it together with a Corollary of
Theorem \ref{existence} to
remove the dimensional constraint and prove Theorem \ref{anyd}. Finally,
we
conclude Section 5 by sketching the proof of Theorem \ref{global}.

\section{The Dirichlet problem for the energy functional
$\mathcal{I}_\Omega(\cdot)$}

\subsection{Notation.} Throughout this paper we denote by $\mathbb{S}^n$ the standard unit sphere in $\mathbb{R}^{n+1}$ and by $\mathbb{S}^{n}_+$ the upper
hemisphere. We use $\textrm{div}_z$ and $\nabla$ to denote respectively the
divergence and the covariant gradient on $\mathbb{S}^n$. Also, we let $\textbf{e}$ be the unit vector in the positive $x_{n+1}$
direction in $\R^{n+1}$ and

$$y=\textbf{e}\cdot z, \quad \text{for} \  z \in \s^n,$$

\medskip

\noindent where `$\cdot$' denotes the Euclidean inner
product in $\R^{n+1}.$

\medskip

We recall the following fact, which will be used in the proof of the existence of minimizers in the next subsection.

\begin{rem}\label{horosphere} Assume $|H|<1$. Let $B_R(a)$ be a ball of radius $R$
centered at $a=(a',-HR) \in \R^{n+1}$ where $a' \in \R^n.$ Then
$S=\partial B_R(a) \cap \h^{n+1}$ has constant hyperbolic mean
curvature $H$ with respect to its outward normal. Analogously, let
$B_R(b)$ be a ball of radius $R$ centered at $b=(b',HR) \in
\R^{n+1}$ where $b' \in \R^n.$ Then $S=\partial B_R(b) \cap
\h^{n+1}$ has constant hyperbolic mean curvature $H$ with respect
to its inward normal.\end{rem}

\subsection{Existence of minimizers.} We now formulate and solve the Dirichlet problem for the functional
$\mathcal{I}_\Omega(\cdot)$ in the Introduction.

Let $\Omega \subset \mathbb{S}^n_+$ ; for a function $v \in
BV(\Omega)$ define,

\begin{eqnarray*}
\int_{\Omega} \sqrt{1+|\nabla v|^2} y^{-n}  :=
\sup\{\int_{\Omega} v(z) \textrm{div}_z[\tilde{\gamma}
y^{-n}] dz + \int_{\Omega} \gamma_{n+1} y^{-n}dz
:\\\nonumber
   \gamma=(\tilde{\gamma}, \gamma_{n+1}) \in C_0^1(\Omega, T\Omega \times
\mathbb{R}), |\tilde{\gamma}|^2_{\mathbb{S}^n} +
|\gamma_{n+1}|^2\leq 1\}.
\end{eqnarray*}

\

\noindent Here we are denoting with $dz$ and
$|\cdot|_{\mathbb{S}^n}$ respectively the measure and the
length on the standard unit sphere.

Let $v \in BV(\Omega)$ and define the energy functional

$$\mathcal{I}_\Omega(v) := \int_{\Omega} \sqrt{1+|\nabla v|^2} y^{-n}  + nH\int_{\Omega}v(z)y^{-(n+1)}dz, $$

\

\noindent where $H$ is a constant with $|H|<1.$ In what follows we
denote

$$A_\Omega(v):= \int_{\Omega} \sqrt{1+|\nabla v|^2} y^{-n}, $$

$$V_\Omega(v) := \int_{\Omega}v(z)y^{-(n+1)}dz.
$$

\

\noindent We omit the subscript $\Omega$ from the definitions above, whenever there is no possibility of confusion.

\medskip

The Dirichlet problem for the energy functional
$\mathcal{I}_\Omega(\cdot)$ consists in minimizing this functional
 among all $v \in BV(\Omega)$ whose trace on
$\partial \Omega$ is a prescribed function $\phi \in
L^1(\partial\Omega).$ However, this problem may not be solvable in
such generality. The following proposition suggests an alternative
form of the Dirichlet problem.

\begin{prop} \label{modifieddirichlet}Assume $\partial \Omega$ is
$C^1$ and let $\phi \in L^1(\partial \Omega).$ Then,

\begin{align*}
&\inf\{\mathcal{I}(v) : v \in BV(\Omega), v=\phi \ \mbox{on} \
\partial \Omega\}=\\ &\inf\{\mathcal{I}(v) + \int_{\partial \Omega}
|v-\phi|y^{-n} dH_{n-1}; v \in BV(\Omega)\}.\end{align*}
\end{prop}

\begin{proof} Let $v \in BV(\Omega)$ and let $\epsilon >0$.
Gagliardo's Theorem (see Theorem 2.16 of \cite{G}) states that there exists a function
$w \in W^{1,1}(\Omega)$ with $w=v-\phi$ on $\partial \Omega$ and

\be\label{smallarea}\int_{\Omega} |\nabla w|y^{-n} \leq
(1+\epsilon)\int_{\partial \Omega}
|v-\phi|y^{-n} dH_{n-1},\ee

\be\label{volume}n|H|\int_\Omega |w|y^{-(n+1)} \leq \epsilon \int_{\partial
\Omega}
|v-\phi|y^{-n} dH_{n-1}.\ee
\

\noindent The function $u=v+w$ is in
$BV(\Omega)$ and $u=\phi$ on $\partial \Omega$. Moreover, by
\eqref{smallarea}

\begin{align*}
\int_{\Omega} \sqrt{1+|\nabla u|^2}y^{-n} &\leq
\int_{\Omega} \sqrt{1+|\nabla v|^2}y^{-n} + \int_{\Omega}
|\nabla w|y^{-n}\\ &\leq \int_{\Omega} \sqrt{1+|\nabla
v|^2}y^{-n}  + (1+\epsilon)\int_{\partial \Omega}
|v-\phi|y^{-n} dH_{n-1}.\end{align*}

\

\noindent Thus, by \eqref{volume}

\begin{align*}\mathcal{I}(u) \leq \mathcal{I}(v) + (1+2\epsilon)
\int_{\partial \Omega}
|v-\phi|y^{-n} dH_{n-1}.\end{align*}

\medskip

\noindent  As $\epsilon$ tends to zero,
taking the infimum over all $v \in BV(\Omega)$ we obtain

\begin{align*}
&\inf\{\mathcal{I}(v) : v \in BV(\Omega), v=\phi \ \mbox{on} \
\partial \Omega\}\leq \\ &\quad \inf\{\mathcal{I}(v) + \int_{\partial
\Omega} |v-\phi|y^{-n} dH_{n-1}; v \in BV(\Omega)\},\end{align*}

\

\noindent which
suffices as the opposite inequality is trivial.
\end{proof}

Proposition \ref{modifieddirichlet} suggests the introduction of
the modified energy functional

$$\mathcal{I}_\Omega^\phi(v) = \mathcal{I}(v) + \int_{\partial \Omega} |v -\phi|y^{-n} dH_{n-1}.$$

\

\noindent Again the dependence on $\Omega$ will be made explicit only when strictly necessary.

\

A compactness argument allows us to conclude that the minimization
problem for $\mathcal{I}^\phi(\cdot)$ is always solvable in
the appropriate class of functions. Precisely we have the
following Theorem. 



\begin{thm}\label{existmin} Assume $\partial \Omega$ is Lipschitz continuous and let $\phi \in
L^{\infty}(\partial\Omega).$  Then, $\mathcal{I}^\phi(\cdot)$
attains its minimum $u $ in $BV(\Omega)$. Moreover $u \in
L^{\infty}(\Omega)$ and $\|u\|_{L^\infty} \leq M$ for some $M = M(\|\phi\|_\infty).$
\end{thm}

\begin{proof}Let $S_\delta:= \{y > \delta\} \cap \mathbb{S}^n_{+}$ contain $\overline{\Omega}$
and let us extend $\phi$ to a $W^{1,1}$ function in
$S_\delta \setminus \overline{\Omega}$ that we will
still denote by $\phi.$ Let $v \in BV(\Omega)$ and define

$$v_\phi = \left\{%
\begin{array}{ll}
    v(z), & \hbox{$z\in \Omega;$} \\
    \phi, & \hbox{$z \in S_\delta \setminus \Omega.$} \\
\end{array}%
\right.    $$

\

\noindent Then, $v_\phi \in BV(S_\delta)$ and by the
trace formula

\begin{align*}
\int_{S_\delta}\sqrt{1+|\nabla v_\phi|^2}y^{-n} &= \int_{\Omega}\sqrt{1+|\nabla v|^2}y^{-n} +
\int_{S_\delta \setminus \Omega}\sqrt{1+|\nabla \phi|^2}y^{-n}\\ &\quad + \int_{\partial
\Omega}|v-\phi|y^{-n}dH_{n-1}.\end{align*}

\

\noindent Therefore,

$$\mathcal{I}_{S_\delta}(v_\phi)=\mathcal{I}^\phi_\Omega(v )+ C(\phi),$$

\

\noindent where
$C(\phi)$ is a constant independent of $v$. Hence in order to
minimize $\mathcal{I}^\phi_\Omega(\cdot)$ among all $BV(\Omega)$
functions, it suffices to minimize
$\mathcal{I}_{S_\delta}(\cdot)$ among all functions $u
\in BV(S_\delta)$, coinciding with $\phi$ in
$S_\delta \setminus \overline{\Omega}.$

Let $\overline{\varphi}$ and $\underline{\varphi}$ be smooth
solutions to the equation

\begin{equation}\label{x}
\textrm{div}_{\mathbb{S}^n}(\frac{y^{-n}\nabla
v}{\sqrt{1+|\nabla v|^2}}) = nHy^{-(n+1)}, \quad \text{in}
\quad S_{\delta}
\end{equation}

\

\noindent such that

\begin{equation}\label{over}\inf_{S_{\delta}}\overline{\varphi} >
\|\phi\|_{L^{\infty}(S_{\delta})},\end{equation}

\noindent and

\begin{equation}\label{under}\sup_{S_{\delta}} \underline{\varphi}
< -\|\phi\|_{L^{\infty}(S_{\delta})}.\end{equation}

\

\noindent The existence of $\overline{\varphi}$ and
$\underline{\varphi}$ follows from Remark \ref{horosphere} by
choosing $a^{\prime}=0$ for a suitable choice of $R$. Explicitly,

\begin{align*}&{\underline{\varphi}} =
-\|\phi\|_{L^{\infty}(S_{\delta})}+\log{(\sqrt{H^2 y^2
+(1-H^2)}-Hy})~,\\
& \ \\ &\overline{\varphi} =
\underline{\varphi}+2\|\phi\|_{L^{\infty}(S_{\delta})}-\log{(1-H)}~.\end{align*}

\medskip

 Now, let $u_j \in
BV(S_\delta)$ be a minimizing sequence, that is

$$\inf\{\mathcal{I}_{S_\delta}(u) : u\in BV(S_\delta), u=\phi \ \textrm{in} \
S_\delta \setminus \overline{\Omega}\}=
\lim_{j}\mathcal{I}_{S_\delta}(u_j)=I.$$

\

\noindent Let us
approximate the $u_j$'s with smooth functions which we still
denote by $u_j$'s. Set

$$\overline{u}_j = \min\{u_j,\overline{\varphi}\}$$ and compute

\begin{align}\label{compare}
\mathcal{I}_{S_\delta}(u_j) &= \mathcal{I}_{S_{\delta} \cap
\{u_j<\overline{\varphi}\}}(u_j)+\mathcal{I}_{S_{\delta} \cap
\{u_j>\overline{\varphi}\}}(u_j)\\\nonumber &\\\nonumber
&=\mathcal{I}_{S_{\delta}}(\overline{u}_j)-\mathcal{I}_{S_{\delta}
\cap \{u_j>\overline{\varphi}\}}(\overline{u}_j)+
\mathcal{I}_{S_{\delta} \cap
\{u_j>\overline{\varphi}\}}(u_j)\\\nonumber &\\\nonumber
&=\mathcal{I}_{S_{\delta}}(\overline{u}_j)+\int_{S_{\delta} \cap
\{u_j>\overline{\varphi}\}}\left[ y^{-n}\left(\sqrt{1+|\nabla
u_j|^2}-\sqrt{1+|\nabla \overline{u}_j|^2}\right)\right.& \\
\nonumber & \ \\ \nonumber &\left. \quad \quad \quad \quad \quad
\quad \quad \quad \quad \quad \quad \quad  \ \ \ \  \ \ \ \ \ \ \
\ \ \ \ \ \ \ \ \  \ + n H(u_j-
\overline{u}_j)y^{-(n+1)}\right]dz\\\nonumber
\\\label{intbyparts} &\geq \mathcal{I}_{S_{\delta}}(\overline{u}_j)+ \quad \int_{S_{\delta}
\cap \{u_j>\overline{\varphi}\}} \left(\frac{y^{-n}\nabla
\overline{u}_j}{ \sqrt{1+|\nabla \overline{u}_j|^2}}\nabla(u_j -
\overline{u}_j)\right.& \\
\nonumber & \ \\ \nonumber &\left. \quad \quad \quad \quad \quad
\quad \quad \quad \quad \quad \quad \quad  \ \ \ \  \ \ \ \ \ \ \
\ \ \ \ \ \ \ \ \  \ + nH(u_j -
\overline{u}_j)y^{-(n+1)}\right)dz.
\end{align}

\medskip

\noindent After integration by parts the integral in
\eqref{intbyparts} is identically zero in view of the fact that
$\overline{\varphi}$ satisfies \eqref{x}-\eqref{over}. Hence,

\begin{equation}\label{1}\mathcal{I}_{S_\delta}(u_j)  \geq
\mathcal{I}_{S_\delta}(\overline{u}_j).
\end{equation}

\

\noindent Analogously, set

$$\underline{u}_j=\max\{\underline{\varphi},
\overline{u}_j\}$$

\

\noindent and note that

$$\underline{\varphi} \leq \underline{u}_j \leq \overline{\varphi}.$$

\ Then,

\begin{eqnarray*}\mathcal{I}_{S_\delta}(\underline{u}_j) &=&
\mathcal{I}_{S_{\delta} \cap \{\overline{u}_j<\underline{\varphi}\}}(\underline{u}_j)+
\mathcal{I}_{S_{\delta} \cap \{\overline{u_j}>\underline{\varphi}\}}(\underline{u}_j) \\
&& \\
&=&
\mathcal{I}_{S_{\delta}}(\overline{u}_j)-\mathcal{I}_{S_{\delta}
\cap \{\overline{u}_j<\underline{\varphi}\}}(\overline{u}_j)+
\mathcal{I}_{S_{\delta} \cap \{\overline{u}_j<\underline{\varphi}\}}(\underline{\varphi}) \\
&& \\
&=&\mathcal{I}_{S_{\delta}}(\overline{u}_j)+\int_{S_{\delta} \cap
\{\overline{u}_j<\underline{\varphi}\}}\left[
y^{-n}\left(\sqrt{1+|\nabla \underline{\varphi}|^2}
-\sqrt{1+|\nabla \overline{u}_j|^2}\right)\right.\\ & \ & \  \\ &
\ & \left. \ \ \quad \quad \ \ \ \ \ \ \ \ \ \ \ \ \ \ \quad \quad
\ \ \ \ \ \ \quad \ \ \ \ \ \ \ \ \ \ \ \ \ \ \ \ +
 nH(\underline{\varphi}-u_j)y^{-(n+1)}\right]dz  \\
&&\\
 &\leq& \mathcal{I}_{S_{\delta}}(\overline{u}_j)+ \int_{S_{\delta} \cap \{\overline{u}_j<\underline{\varphi}\}}
\left(\frac{y^{-n}\nabla \underline{\varphi}}{ \sqrt{1+|\nabla
\underline{\varphi}|^2}}\nabla(\underline{\varphi}- {u}_j)
\right.\\ & \ & \  \\ & \ & \left. \ \ \quad \quad \ \ \ \ \ \ \ \
\ \ \ \ \ \ \quad \quad \ \ \ \ \ \ \quad \ \ \ \ \ \ \ \ \ \ \ \
\ \ \ \ + nH(\underline{\varphi}-u_j)y^{-(n+1)}\right)dz.
\end{eqnarray*}

\medskip

 \noindent Again since $\underline{\varphi}$ satisfies
\eqref{x}-\eqref{under}, the last term  vanishes. Hence,

\begin{equation}\label{2}\mathcal{I}_{S_\delta}(\overline{u}_j)
\geq \mathcal{I}_{S_\delta}(\underline{u}_j).
\end{equation}

\

Combining \eqref{1} and \eqref{2} we obtain that

$$\lim_j
\mathcal{I}_{S_\delta}(u_j) \geq K,$$

\

\noindent for some constant
$K$, hence $I$ is finite. Moreover, the $\underline{u}_j$'s are
are uniformly bounded in $BV$
(since  $\mathcal{I}_{S_\delta}(\underline{u}_j)\leq
 \mathcal{I}_{S_\delta}(u_j) \leq C,$ as the $u_j$'s are a minimizing sequence)
and we can extract a subsequence which converges in
$L^1(S_\delta)$ to some function $u \in BV(S_\delta)$. Furthermore
$u \in L^{\infty}(\Omega)$ and $u=\phi$ in $S_\delta \setminus
\overline{\Omega}$. Then by the lower semicontinuity of our
functional we find that $u$ is the required minimizer.
\end{proof}

We now collect a few more facts about minimizers, which will be used in
the next sections.

\begin{rem}\label{unique} From the strict convexity of our functional, in
particular

$$\frac{\mathcal{I}^\phi(v_1) + \mathcal{I}^\phi(v_2)}{2}
\geq \mathcal{I}^\phi\left(\frac{v_1+v_2}{2}\right),$$

\

\noindent we obtain that if $v_1,v_2$
are two minima of $\mathcal{I}^\phi(\cdot)$, then $v_1=v_2+$const.
Moreover on $\partial \Omega$ the traces of $v_1$ and $v_2$ satisfy $(v_1-\phi)(v_2-\phi)>0$. Finally,
if $v_1$ and $v_2$ have the same trace $\phi$ on $\partial
\Omega$ then $v_1=v_2.$ \end{rem}

\begin{cor}\label{maxprinc}Let $v$ minimize $\mathcal{I}^\phi(\cdot)$,  and  $\phi \in C(\partial \Omega).$
Assume that $\overline{\varphi}$ (resp. $\underline{\varphi}$ ) is a smooth
supersolution (resp. subsolution) to equation $\eqref{dirichlet1}$, with
$\overline{\varphi} \geq \phi$ (resp. $\underline{\varphi} \leq \phi$) on $\partial \Omega.$
Then $\overline{\varphi} \geq v$ (resp. $\underline{\varphi} \leq v$) in
$\Omega$.\end{cor}

\medskip

The corollary above follows by the same argument as in the proof of Theorem
\ref{existmin} (in particular see formula \eqref{compare}), together with
Remark \ref{unique}.

\begin{lem}\label{uniqtrace} Let $v_i \in BV(\Omega)$ minimize
$\mathcal{I}^{\phi_i}, \phi_i \in L^\infty(\partial \Omega)$, $v_i=\phi_i$ on $\partial \Omega$ (in the trace sense) $i=1,2$. Assume $\phi_1 \geq \phi_2$ on $\partial \Omega$. Then $v_1 \geq v_2$ in $\Omega$.
\end{lem}

\begin{proof} Set
$$v_{max}= \max\{v_1,v_2\}, \quad
v_{min}=\min\{v_1,v_2\}.$$

\smallskip

\noindent Then,
\be\label{areaineq} \mathcal{I}(v_1) + \mathcal{I}(v_2) \geq
\mathcal{I}(v_{max}) + \mathcal{I}(v_{min}).\ee

\medskip

\noindent Indeed formula \eqref{areaineq} clearly holds in the
case when $v_1$ and $v_2$ are smooth. We can then approximate
$v_i$, $i=1,2$ with a sequence $\{v_i^m\}$ of smooth functions
such that $v_i^m \rightarrow v_i$ in $L^1$ and $A(v_i^m)
\rightarrow A(v_i).$ Then by the lower semicontinuity of our
functional we immediately get \eqref{areaineq} for $BV$ functions.

Moreover, since $\phi_1 \geq \phi_2$ on $\partial \Omega$, we have that
$v_{max}$ has the same
trace as $v_1$ while $v_{min}$ has the same trace as $v_2$ on $\partial
\Omega$. The desired
claim now follows by the
uniqueness of minimizers (Remark \ref{unique}).
\end{proof}

\smallskip

\begin{rem}\label{sol=min}It is straightforward to show that
smooth
solutions to the Dirichlet problem for the divergence equation
\eqref{dirichlet1} on
$\Omega$ and boundary data $\phi$, also
minimize the energy
integral $\mathcal{I}(\cdot)$
among all competitors equal to $\phi$ on $\partial \Omega$.
\end{rem}

\section{Regularity in low dimensions}

In this section we prove our main regularity result Theorem \ref{lowd}.
The
existence of local bounded minimizers is guaranteed by Theorem
\ref{existmin}.

We proceed to investigate the connection between non-parametric
and parametric surfaces of constant mean curvature in
hyperbolic space.

For any function $v$ over $\Omega$ we set

$$V:= \{x \in \mathbb{R}^{n+1}: x= e^wz, z\in \Omega, -\infty < w <
v(z)\}.$$

\

\noindent $V$ is the subgraph of the radial graph
defined by

$$X= e^{v(z)}z, \quad z \in \Omega.$$

\


\noindent Also, for any $T>0$, we define

\begin{align*}&C_T:=\{x \in \mathbb{R}^{n+1}: x= e^wz, z\in \Omega,
-T-1 < w < T+1\},\\
\ \\
&\underline{C}_{T}:=\{x \in \mathbb{R}^{n+1}: x= e^wz, z\in \Omega, -T -1 < w <
-T\},\\
\ \\
&\overline{C}_{T}:=\{x \in \mathbb{R}^{n+1}: x= e^wz, z\in \Omega, -T-1 < w <
T\}.\end{align*}




\

\noindent Let us denote by

$$\mathcal{E} := \{E \subseteq C : E \ \text{measurable}, \ \underline{C}_{T}\subseteq E \subseteq \overline{C}_{T} \}.$$


\medskip

\noindent Also, let us define the set functionals representing respectively the
perimeter and the volume in $C_T$ of a set $U$ in the hyperbolic
space $\mathbb{H}^{n+1}$:

$$\mathcal{P}_{C_T}(U) := \sup\{\int_{C_T\overline{}} \varphi_U\textrm{div}_x({g x_{n+1}^{-n}})dx : g \in C^{1}_0(C_T;\mathbb{R}^{n+1}),
|g|^2 \leq 1\}$$

$$\textit{Vol}(U) := \int_{U\cap C_T} x_{n+1}^{-(n+1)}dx.$$

\

\noindent Here we denote by $\varphi_U$ the characteristic
function of a set $U$. We will often drop the subscript $C_T$, whenever this generates no confusion.

Set,

$$\mathcal{F}(U) = \mathcal{P}(U) + nH \it{Vol}(U).$$

\

We wish to prove the following theorem.

\begin{thm}\label{subgraphmin} Let $v \in BV(\Omega) \cap L^\infty(\Omega)$ be a
local minimizer of $\mathcal{I}(\cdot)$ and let $T > \|v\|_{L^\infty}.$ Then $V_T:=V \cap C_T$ locally minimizes
$\mathcal{F}(\cdot)$ among all competitors in $\mathcal{E}.$
\end{thm}

We start with the following proposition.

\begin{prop}\label{subgr} Let $v \in BV(\Omega) \cap L^\infty(\Omega)$, and let $T > \|v\|_{L^\infty}.$ Then,

\begin{equation}\label{equality}\mathcal{F}(V_T) = \mathcal{I}(v) + k(T+1)\end{equation}

\

\noindent where $V_T:= V \cap C_T$ and
$k=\int_{\Omega} y^{-(n+1)}dz.$
\end{prop}

\begin{proof}
We start by showing that

$$\mathcal{P}(V_T)\geq A(v).$$

\

\noindent By
definition, for any $g$ compactly supported in $C_T$ satisfying
$|g|^2 \leq 1$,we have

\begin{align}
\mathcal{P}(V_T) &\geq \int_{V_T} \textrm{div}_x[x_{n+1}^{-n}g(x)]dx=\\
&= \int_{\Omega} \int_{-T-1}^{v(z)} \textrm{div}_{z,w}[y^{-n}g(z,w)]dwdz\end{align}

\

\noindent where in the second line we performed the change of variable $x =
e^w z$. Also we denote by $\textrm{div}_{z,w}$ the divergence on
the manifold $\mathbb{S}^n \times \mathbb{R}$ with the standard
product metric. Notice that $g(z,w) = (\tilde{g}(z,w),
g_{n+1}(z,w))$ satisfies $|\tilde{g}|_{\mathbb{S}^n}^2 +
|g_{n+1}|^2 \leq 1$. Since $g$ is arbitrary, we can choose

$$(\tilde{g}(z,w), g_{n+1}(z,w))= (\tilde{\gamma}(z),
\gamma_{n+1}(z))\eta(w),$$

\

\noindent where $\gamma = (\tilde{\gamma},
\gamma_{n+1})$ is a vector field compactly supported on $\Omega$
such that $|\tilde{\gamma}|^2_{\mathbb{S}^n} + |\gamma_{n+1}|^2
\leq 1$, while $\eta$ is compactly supported in $[-T-1,
\sup_{\Omega}v+1]$ and such that $\eta \equiv 1$ on $[-T,
\sup_{\Omega}v]$ and $|\eta| \leq 1.$ Thus,

\begin{align*}
\mathcal{P}(V_ T) &\geq \int_{\Omega} \int_{-T-1}^{v(z)}
\textrm{div}_{z,w}[\gamma\eta y^{-n}]dwdz=\\
&\int_{\Omega}\int_{-T-1}^{v(z)}\textrm{div}_z[\tilde{\gamma}y^{-n}]\eta(w)dwdz +
\int_{\Omega}\int_{-T-1}^{v(z)} \gamma_{n+1}(z)y^{-n}\eta'(w)dwdz.\end{align*}

\

\noindent From our choice of $\eta$ we have that

$$\int_{-T-1}^{v(z)} \eta'(w)dw =1$$
and

$$\int_{-T-1}^{v(z)}\eta(w)dw = v(z) - c$$ with $c$ constant.

Thus,

$$\mathcal{P}(V_T) \geq \int_{\Omega} v(z)\textrm{div}_z[\tilde{\gamma} y^{-n}]dz + \int_{\Omega} \gamma_{n+1}y^{-n} dz,$$

\

\noindent and the desired statement follows by taking the sup over all
$\gamma=(\tilde{\gamma}, \gamma_{n+1})$ of length smaller than 1,
compactly supported in $\Omega$.

The opposite inequality follows by a standard limiting argument.
In the case when $v \in C^1(\Omega)$ then clearly

$$\mathcal{P}(V_T) = A(v).$$

\

\noindent Now let $v_j \in C^{\infty}(\Omega), v_j \rightarrow v$
in $L^{1}(\Omega)$ and $A(v_j) \rightarrow
A(v).$ Then $V_{j,T} \rightarrow V_{T}$ in $L^1(C)$ and
therefore by the lower semicontinuity of the perimeter functional
we get

$$\mathcal{P}(V_T) \leq \liminf_{j\rightarrow \infty} \mathcal{P}(V_{j,T})=
\lim_{j\rightarrow \infty} A(v_j) = A(v).$$

\

Finally, we compute

$$\textit{Vol}(V_T)= \int_{V_T} x_{n+1}^{-(n+1)}dx =
\int_{\Omega}\int_{-T-1}^{v(z)} y^{-(n+1)}dwdz=
V(v) + k(T+1),$$

\

\noindent which concludes the proof.
\end{proof}

Let $E \in \mathcal{E}$ and denote
by $\tilde{E}$ the image of $E$ under the coordinate
transformation $x=e^wz, z\in \Omega, -T -1< w < T
+1$. Set

\begin{equation}\label{rearrf}u(z) = \int_{-T}^{T} \varphi_{\tilde{E}}(z,w) dw  - T , \quad z\in \Omega.\ee

\

\noindent The subgraph in $C_T$ of the radial surface $X=e^{u(z)}
z, z\in \Omega$ is the rearrangement of the set $E$ in the radial
direction.

\begin{prop}\label{rearrange} For any $E \in \mathcal{E}$ we have,

\begin{equation}\label{inequal}\mathcal{F}(E) \geq \mathcal{I}(u) + k(T+1),
\end{equation}

\

\noindent where
$k=\int_{\Omega} y^{-(n+1)}dz.$

\end{prop}

\begin{proof}
According to the definition,

$$ \mathcal{P}(E) \geq \int_{C}
\varphi_E \textrm{div}_x[x_{n+1}^{-n}g(x)]dx
= \int_{\Omega} \int_{-T-1}^{T} \varphi_{\tilde{E}}\textrm{div}_{z,w}[y^{-n}g(z,w)]dwdz$$

\

\noindent after performing the change of variable $x =
e^w z$. As in Proposition \ref{subgr} since $g$ is arbitrary, we
can choose

$$(\tilde{g}(z,w), g_{n+1}(z,w))= (\tilde{\gamma}(z),
\gamma_{n+1}(z))\eta(w),$$

\

\noindent where $\gamma = (\tilde{\gamma},
\gamma_{n+1})$ is a vector field compactly supported on $\Omega$
such that $|\tilde{\gamma}|^2_{\mathbb{S}^n} + |\gamma_{n+1}|^2
\leq 1$, while $\eta$ is compactly supported in $[-T-1,
T+1]$ and such that $\eta \equiv 1$ on $[-T, T]$ and
$|\eta| \leq 1.$ Thus,

\begin{align*}
\mathcal{P}(E) &\geq \int_{\Omega} \int_{-T-1}^{T}\varphi_{\tilde{E}}
\textrm{div}_{z,w}[\gamma\eta y^{-n}]dwdz=\\
&\int_{\Omega}\int_{-T-1}^{T}\varphi_{\tilde{E}}\textrm{div}_z[\tilde{\gamma}y^{-n}]\eta(w)dwdz +
\int_{\Omega}\int_{-T-1}^{T} \varphi_{\tilde{E}}\gamma_{n+1}(z)y^{-n}\eta'(w)dwdz.\end{align*}

\

From our choice of $\eta$ we have that

$$\int_{-T-1}^{-T} \eta'(w)dw =1,$$

\

\noindent and also $\varphi_{\tilde{E}} (z,w) \equiv
1$ for $-T-1 < w < -T.$ Thus, according to the definition
of $u$ we have

$$\mathcal{P}(E) \geq \int_{\Omega}
u(z)\textrm{div}_z[\tilde{\gamma} y^{-n}]dz + \int_{\Omega}
\gamma_{n+1}y^{-n} dz,$$

\

\noindent and the desired statement follows
by taking the sup over all $\gamma=(\tilde{\gamma}, \gamma_{n+1})$
of length smaller than 1, compactly supported in $\Omega$.

Finally, we compute

\begin{eqnarray}\label{preservevol}\textit{Vol}(E)= \int_{C} \varphi_{E} x_{n+1}^{-(n+1)}dx =
\int_{\Omega} (u(z) + T)y^{-(n+1)}dz + \\\nonumber
\int_{\Omega}\int_{-T-1}^{-T}y^{-(n+1)}dwdz =
V(u) + k(T+1),\end{eqnarray}

\

\noindent  which concludes the proof.
\end{proof}

We are now ready to prove our Theorem.

\

\textit{Proof of Theorem \ref{subgraphmin}.} Let $A \subset
\subset \Omega$ and let $E \in \mathcal{E}$ coincide with $V_T$ outside a compact set in $\{x \in
\mathbb{R}^{n+1} : x= e^wz, z\in A, -T-1 < w < T +
1\}.$ Then the function $u$ associated to $E$ coincides
with $v$ outside of $A$ and hence according to \eqref{equality}
and \eqref{inequal},

\begin{eqnarray*}
\mathcal{F}(V_T) \leq \mathcal{I}(v) + k(T+1) \leq
\mathcal{I}(u) + k(T+1) \leq \mathcal{F}(E).
\end{eqnarray*} \qed

\medskip

Since $V_T$ locally minimizes $\mathcal{F}$ in $\mathcal{E}$, it
is known that the boundary of $V_T$ is a regular (analytic)
hypersurface outside a closed set $S$, with $H_{n-6}(S) =0$ (see
\cite{S}). As an immediate corollary we shall prove that $v$ is
regular in $L = \Omega \setminus \textrm{proj}_\Omega S.$

Towards this aim, we need to recall the following
lemma that can be found in \cite{GS}.

\begin{lem}\label{radiality} Let $\Sigma$ be a constant mean curvature
hypersurface in $\mathbb{H}^{n+1}$ with position vector $X$ in $\R^{n+1}$ and unit normal $\nu$ with respect to the Euclidean metric. Let $|A|$ and $\Delta$ denote respectively the norm of the second fundamental
form of $\Sigma$ and the Laplace-Beltrami operator on $\Sigma$ with respect to the hyperbolic metric. Then,

\begin{equation}\label{elliptic}\Delta\frac{X\cdot\nu}{u}= (n-|A|^2)\frac{X\cdot\nu}{u},\end{equation}
\end{lem}
\noindent \textit{where $u$ denotes the height function} $u=X \cdot \textbf{e}.$

\

\begin{cor} Let $v \in BV(\Omega) \cap L^\infty(\Omega)$ be a local minimizer to
$\mathcal{I}(\cdot).$ Then $v \in C^{\infty}(L)$ with
$H_{n-6}(\Omega \setminus L)=0.$
\end{cor}

\begin{proof} Let $\Sigma$ be the radial graph associated to $v$. We use the notation from Lemma \ref{radiality}.
Assume by contradiction that $X\cdot \nu =0$ at some point $z \in
L$. Then,

$$X \cdot \nu \geq 0.$$

\smallskip

\noindent Hence
according to \eqref{elliptic} and the strong maximum principle we have

$$X \cdot \nu \equiv
0 \ \text{in} \ L,$$

\smallskip

\noindent which contradicts the analyticity of the graph of $v$
outside of the singular set $S.$
\end{proof}

Theorem \ref{lowd} is a straightforward consequence of the
Corollary above.

\

Using Propositions \ref{subgr} and \ref{rearrange}, we can
also prove the following uniqueness result which will be used in
the next section. First we set some notation, to which we will refer later.

Let $\overline{v}\geq\underline{v}$ be continuous functions on
$\Omega$ with $\overline{v}=\underline{v}=\varphi$ on $\partial
\Omega$, $|\overline{v}|,|\underline{v}| \leq T$. Denote by
$\overline{V}, \underline{V}$ respectively the subgraphs in $C_T$ of the radial surfaces
$X=e^{\overline{v}}z$, and $X=e^{\underline{v}}z$, $z \in \Omega$. Let

$$\mathcal{V}:=\{E \subseteq C_T: E \ \text{measurable}, \  \underline{V} \subseteq E \subseteq \overline{V} \}.$$

\smallskip

\begin{lem}\label{minuniq}The minimization problem for $\mathcal{F}(\cdot)$ in the class $\mathcal{V}$ admits a unique solution $E$. Moreover,
$\partial E$ is a radial graph over $\s_+^n$.\end{lem}

\begin{proof}Let $E_1$ and $E_2$ be distinct minimizer of $\mathcal{F}$ in
$\mathcal{V}.$ Using (see for example \cite{G}, Lemma 15.1)
$$\mathcal{P}_{C_T}(E_1 \cap E_2) + \mathcal{P}_{C_T}(E_1 \cup E_2) \leq
\mathcal{P}_{C_T}(E_1) + \mathcal{P}_{C_T}(E_2)$$

\smallskip

\noindent we obtain that $E_1 \cap E_2$, $E_1 \cup E_2$ also
minimize $\mathcal{F}$ in the same class. Denote by $u_1$, $u_2$
be the associated rearrangement functions (given by formula
\eqref{rearrf}) for these minimizers. Notice that $u_1 \neq u_2$.
Indeed $E_1 \neq E_2$ implies that $E_1 \cap E_2$ has smaller
volume than $E_1 \cup E_2$, and the volume is preserved by the
rearrangements up to an additive constant (see
\eqref{preservevol}). Then according to Propositions \ref{subgr}
and \ref{rearrange}, $u_1$ and $u_2$ minimize
$\mathcal{I^{\varphi}}$ in the class of all competitors $v$ with
$\underline{v} \leq v \leq {\overline{v}}.$ Since $(u_1 + u_2)/2$
is in the same class, we can apply the same convexity argument as
in Remark \ref{unique} to conclude that $u_1 = u_2$. Thus, we
reached a contradiction.
\end{proof}

\section{The  Dirichlet problem with smooth boundary data.}

In this section we show that upon assuming the right condition on the boundary of $\Omega$, it is possible to set up and solve the Dirichlet problem for the energy functional $\mathcal{I}(\cdot)$ in the classical sense, that is finding a smooth minimizer $v$ among all competitors with the same smooth boundary data. This result is of independent interest. Moreover, a corollary of this result, together with the gradient bound presented in the next section will allow us to remove the dimensional constraint of Theorem \ref{lowd} and prove the interior smoothness of bounded $BV$ minimizers in any dimension.

Precisely we prove the following result.

\begin{thm}\label{existence2} Let $\Omega$ be a subdomain of $\mathbb{S}^n_+$  with $\partial \Omega \in C^2$, and let $\gamma$ be a $C^2$ radial graph over $\partial \Omega$. Let $h$ be  the hyperbolic mean curvature  of the radial cone over $\partial \Omega$ restricted to  $\partial \Omega$.  Then if $h> |H|$ , there exists a unique smooth radial graph
$\Sigma$  of constant mean curvature H in $\h^{n+1}$ (defined over $\Omega$)  with  boundary $\gamma.$
 \end{thm}

Theorem \ref{existence} follows by standard elliptic theory, combining Theorem \ref{existence2} and the interior gradient bound Proposition \ref{th2} in the next section.

We first need some preliminaries. Let $\Sigma$ be an hypersurface in $\mathbb{H}^{n+1}$ and let $X$ be the position vector
of $\Sigma$ in $\R^{n+1}$. We set $\textbf{n}$ to be a global unit normal vector
field to $\Sigma$ with respect to the hyperbolic metric. This
determines a unit normal $\nu$ to $\Sigma$ with respect to the
Euclidean metric by the relation

$$\nu = \frac{\textbf{n}}{u},$$

\

\noindent where $u$ denotes the height function $u=X \cdot \textbf{e}.$
The hyperbolic principal curvatures $\kappa_1, \ldots \kappa_n$ of $\Sigma$ (with respect to
$\textbf{n}$) are related to the Euclidean principal curvatures  $\tilde{\kappa}_1, \ldots
\tilde{\kappa}_n $ of $\Sigma$ (with respect to $\nu$)
by the well-known formula

\[\kappa_i=u \tilde{\kappa}_i+\nu^{n+1}~.\]

\medskip

\noindent Therefore the hyperbolic mean curvature $H$ and Euclidean mean curvature $H_{E}$ are related
by

\be \label{hmc}
H=uH_{E}+\nu^{n+1}.
\ee

\medskip

Let $\tau_1,\ldots,\tau_n$ be a local frame of smooth vector
fields on $\s^n_+$.
Denote by $\sigma_{ij}=\tau_i\cdot\tau_j$ the
standard metric on $\s^n$ and $\sigma^{ij}$ its inverse.
For a function $v$ on $\s^n$, we use the notation $v_i=\nabla_i v=
\nabla_{\tau_i}v~,~v^i=\sigma^{ik}v_k~,~  v_{ij}=\nabla_j\nabla_iv,$ etc.

For a radial graph $X=e^v z$, the induced Euclidean metric and its inverse are given by

\be \label{metric} \tilde{g}_{ij}=e^{2v}(\sigma_{ij}+v_i
v_j)~,~\tilde{g}^{ij}=e^{-2v}\left(\sigma^{ij}-\frac{v^i
v^j}{W^2}\right), \ee

\smallskip

\noindent where

 \be\label{W}  W=\sqrt{1+|\nabla v|^2}.\ee

 \

\noindent The outward unit normal to $X$ is

\be \label{normal} \nu=\frac{z-\nabla v}{W}, \ee

\

\noindent and
the Euclidean second fundamental form is given by

\[\tilde{b}_{ij}=\frac{e^v}W(v_{ij}-v_i v_j-\sigma_{ij}).\]

\

\noindent Therefore, using \eqref{metric} we have

\be \label{emc} nH_{E}=\tilde{g}^{ij}
\tilde{b}_{ij}=\frac{e^{-v}}W\left\{\left(\sigma^{ij}-\frac{v^i
v^j}{W^2}\right)v_{ij}-n\right\}. \ee

\

\noindent Combining \eqref{hmc}, \eqref{normal}, \eqref{emc}, we have
\begin{lem} \label{equation}The radial graph $X=e^v z$ has constant hyperbolic mean curvature $H$ if and only if  $v$ satisfies the nondivergence form elliptic equation\end{lem}

\be \label{local} \frac1W
a^{ij}v_{ij}=\frac{n}y\left(H+\frac{\textbf{e}\cdot \nabla
v}W\right),~a^{ij}=\sigma^{ij} - \frac{v^i v^j}{W^2}.\ee

\

It is easily seen that \eqref{local} can be written in divergence form as

\begin{equation}\label{dirichlet}\textrm{div}_{z}\left(\frac{y^{-n}\nabla{v}}{\sqrt{1+|\nabla v|^2}}\right)=
nHy^{-(n+1)}, \end{equation}

\

\noindent which is the Euler-Lagrange equation of our functional
\eqref{functional}, the usual area plus $nH$ volume functional for the
hyperbolic radial graph.

Given a subdomain $\Omega$ of $\mathbb{S}^{n}_+$ we can then formulate
(according to Lemma \ref{equation}) the following Dirichlet problem for a
radial graph $X=e^v z$ over $\Omega$ of constant hyperbolic mean curvature
$H$,

\begin{eqnarray}\displaystyle\label{problem}
& \dfrac1W \left(\sigma^{ij} -\dfrac{v^i v^j}{W^2}\right)v_{ij}=\dfrac{n}y\left(H+\dfrac{\textbf{e}\cdot \nabla v}W\right)\quad \text{in} \ \Omega, \\\nonumber & \
\\\label{bd}
& v = \phi \quad \text{on} \ \partial \Omega.\\\nonumber
\end{eqnarray}

\begin{rem} An equivalent problem has been studied (even for prescribed
mean curvature) by Nitsche \cite{N} using a more complicated model of
hyperbolic space. However as we shall see below, the problem can be easily
solved directly, even for continuous boundary data. \end{rem}

Theorem \ref{existence2}, which is an existence and uniqueness statement
for the Dirichlet problem \eqref{problem}-\eqref{bd}, will follow from the
following result, by standard elliptic theory.

\begin{thm}\label{bigdomain} Let $h$ be the hyperbolic mean curvature of
the radial cone $C$ over $\partial \Omega ~\mbox{restricted to $\partial
\Omega$},$ and let $\phi \in C^2(\s^n_+).$ Then if $h> |H|$, there exists
a unique minimizer $v$ of $\mathcal{I}(\cdot)$ in $C^{0,1}(\Omega)$
such that $v=\phi$ continuously on $\partial \Omega$.
 \end{thm}

The main ingredient it the proof of Theorem \ref{bigdomain} is the
following proposition which guarantees the existence of lower and upper
barriers. The existence of such barriers can be obtained in a
straightforward way using the method of \cite{Sp}. We will sketch the main
steps of the proof.

\begin{prop}\label{bar} Let $\phi \in C^2(\s^n_+)$ and assume the
solvability condition of Theorem $\ref{bigdomain}$, Then the  Dirichlet problem
\eqref{problem}-\eqref{bd} admits lower and upper barriers.
\end{prop}

First, we recall the definition of barriers. Let $\phi$ be a Lipschitz
continuous function on $\partial \Omega$. For $z \in \Omega$,
denote by $d(z)$ the distance of $z$ from $\partial \Omega$ in the spherical metric.

An upper barrier $\overline{v}$ relative to the Dirichlet problem
\eqref{problem}-\eqref{bd} in $\Omega$ is a Lipschitz continuous
function defined in a neighborhood $N_{\delta}=\{z \in \Omega  :
d(z)< \delta\}$ of $\partial \Omega$, such that $\overline{v}$ is
a supersolution in $N_\delta$ and

\be\label{barriercond} \overline{v}=\phi \ \text{on} \ \partial
\Omega; \quad \overline{v} \geq \sup_{\partial \Omega}\phi \
\text{on} \
\partial N_\delta \cap \ \Omega. \ee

\

\noindent Analogously, one can define a lower barrier $\underline{v}$ as a
subsolution in $N_\delta$ such that

\be \underline{v}=\phi \ \text{on} \ \partial \Omega; \quad
\underline{v} \leq \inf_{\partial \Omega}\phi \ \text{on} \
\partial N_\delta \cap \ \Omega. \ee

\begin{rem}\label{idea}  Let $N$ be the interior unit normal (in the metric of the sphere) to $\partial \Omega.$ Then the Euclidean mean curvature $h_{E}$ of $C$ restricted to $\partial \Omega$ is given by $h_{E}=\frac {n-1}{n}
 \mathcal{H}_{\partial \Omega}$ and so

 \be \label{formula1}
h=yh_{E}+\textbf{e}\cdot N=\frac{n-1}{n}y\mathcal{H}_{\partial \Omega} + \textbf{e}\cdot N~.
 \ee

\

\noindent  Moreover, if
 $\mathcal{H}_{\partial \Omega}(z)$ denotes the mean curvature at $z$ of the parallel hypersurface  at distance $d(z)$ to
 $\partial \Omega$ passing through $z$, then

 \be \label{formula2}
(n-1)\mathcal{H}_{\partial \Omega}(z)=-\textrm{div}_{z}\nabla d=-\Delta_{z}d(z).
\ee

\

We shall use these formulae in the construction of barriers in
Proposition \ref{bar}, which now follows.
\end{rem}

\textit{Proof of Proposition $\ref{bar}$.} The proof follows the argument of \cite{Sp} and is  similar to the Euclidean case (see for example \cite{GT},\cite{G}.) For completeness, we present
a sketch of the proof.

We proceed to construct an upper barrier $\overline{v}$. According
to the definition of upper barrier and equation \eqref{local} we
need to show that

\begin{equation}
M\overline{v}:=\frac1Wa^{ij}\overline{v}_{ij}-\frac{n}{y}\textbf{e}\cdot
\frac{\nabla \ov{v}}W \leq \frac{nH}y \ \text{in} \ N_{\delta},
\end{equation}

\

\noindent for some $\delta$ to be chosen later.
Here

$$a^{ij}=\sigma^{ij} -
\frac{\ov{v}^i\ov{v}^j}{W^2},$$ and

$$W = \sqrt{1+|\nabla \ov{v}|^2}.$$

\

\noindent Also we must satisfy condition
\eqref{barriercond}. Let us pick

$$ \ov{v}(z)= \phi(z) +
\psi(d(z)),$$

\

\noindent where $\psi$ is a $C^2$ function on $[0,\delta]$
satisfying

\be\label{psi} \psi(0)=0, \psi'(t) \gg 1,
\psi''(t)<0,\ee

\be\label{maxbound}\psi(\delta) \geq 2
\sup_{\Omega}|\phi|=M.\ee

\

Using   $|\nabla d|=1,~d^i d_{ij}=0$ and  $\sigma^{ij}d_{ij}=\Delta d$, \eqref{psi}
and the definition of $a^{ij},$  we find

$$
M\overline{v} \leq
\frac{\psi'}{\sqrt{1+\psi'^2}}(\Delta d-\frac{n}y \textbf{e}\cdot \nabla d)+\frac{\psi''}{(1+\psi'^2)^{\frac32}}+
O(\frac1{y\sqrt{1+\psi'^2}}).
$$

\

\noindent Recalling Remark \ref{idea} we can express this as

\be\label{euclid2}
M\overline{v} \leq
-\frac{\psi'}{y\sqrt{1+\psi'^2}}((n-1)y\mathcal{H}_{\partial \Omega}(z)+n\textbf{e}\cdot N(z))+\frac{\psi''}{(1+\psi'^2)^{\frac32}}+
O(\frac1{y\sqrt{1+\psi'^2}}).
\ee

\

Let $\psi(t) = \frac{1}{K}\log(1+\beta t)$ where $\beta=K^2 e^{MK}$ and $\delta=K^{-2}$. Then \eqref{psi} and
\eqref{maxbound} are satisfied as

$$\psi'(t) \geq \frac{\beta}{K(1+\beta \delta)}>K(1-e^{-MK}),~\psi''=-K\psi'^2, $$ and

$$\psi(\delta) = \frac{1}{K}\log(1+\beta\delta) \geq M.$$

\

\noindent Assume  the strict solvability condition $h\geq|H|+2\epsilon_0$. Then

\be \label{euclid3} (n-1)y\mathcal{H}_{\partial
\Omega}(z)+n\textbf{e}\cdot N(z)\geq n(|H|+\epsilon_0) \ee

\medskip

\noindent in $N_{\delta}$ for small $\delta$. Hence combining \eqref{euclid2}, \eqref{euclid3}

\be \label{euclid4}
M\overline{v} \leq - \frac{\psi'}{\sqrt{1+\psi'^2}}\frac{n(|H|+\epsilon_0)}y-\frac{K\psi'^2}{(1+\psi'^2)^{\frac32}}+
O(\frac1{y\sqrt{1+\psi'^2}}).
\ee

\

\noindent Therefore we can choose $K$ large  so that $\overline{v}$ is an
upper barrier in $N_{\delta}$.  Analogously $\displaystyle\underline{v}=\phi-\frac{1}{K}\log(1+\beta d)$ is
a lower barrier in $N_{\delta}$. \qed

\begin{rem} \label{strictbarrier} Note that when the strict solvability condition $h\geq|H|+2\epsilon_0$ is satisfied, we obtain gradient and continuity estimates on $\partial \Omega$ that are independent of $\min_{\partial \Omega} y $.\end{rem}

\begin{rem} Under certain conditions we can sharpen the solvability condition to $h\geq |H|$.
Suppose $h=|H|$ at $P\in \partial \Omega$ and let

\[nh(s)=(n-1)y(z(s))\mathcal{H}_{\partial \Omega}(z(s))+n\textbf{e}\cdot N(z(s))\]

\medskip

\noindent along the (inward) geodesic orthogonal to $\partial \Omega$ starting at $P$. Note that

$$\dot{y}(s)=\textbf{e}\cdot N(s)$$ $$\ddot{y}(s)=-y(s).$$

\medskip

\noindent Hence from standard comparison theory
(see \cite{Sp})

\begin{align} \label{euclid5}
n\dot{h}(s)&= (n-1)(y \dot{\mathcal{H}}_{\partial \Omega}(s)+\mathcal{H}_{\partial \Omega}(s)\textbf{e}\cdot N(s))-ny(s)\\
&\geq (n-1)(y(s)(\mathcal{H}_{\partial \Omega}^2(s)+1)+\mathcal{H}_{\partial \Omega}(s)\textbf{e}\cdot N(s))
-ny(s). \nonumber
\end{align}

\

\noindent Using $\displaystyle\textbf{e}\cdot N=|H|-\frac{n-1}n y(P)\mathcal{H}_{\partial \Omega}$ at $P$ in \eqref{euclid5} gives

\begin{align*} \label{euclid6}
\frac{n}{n-1}\dot{h}(0)&\geq y(P)\mathcal{H}_{\partial \Omega}^2 +
\mathcal{H}_{\partial \Omega}\left(|H|-\frac{n-1}n
y(P)\mathcal{H}_{\partial \Omega}\right)-\frac1{n-1}y
\nonumber \\
&=\frac{y(P)}n \mathcal{H}_{\partial \Omega}^2 +|H|\mathcal{H}_{\partial \Omega}-\frac1{n-1}y.
\end{align*}

\

Hence if $$\displaystyle \mathcal{H}_{\partial
\Omega}>-\frac{n}2\left(-|H|+\sqrt{H^2+\frac{4y^2}{n(n-1)}}\right),$$

\

\noindent then $\dot{h}(0)>0$ so we obtain from \eqref{euclid2}
(for small $\delta$)

$$
M\overline{v}+\frac{n|H|}y \leq -\frac{K\psi'^2}{(1+\psi'^2)^{\frac32}}+O(\frac1{y\sqrt{1+\psi'^2}})<0
$$

\

\noindent if we choose $K$ large enough (but now depending on $\min_{\partial \Omega}y$).
\end{rem}
\

We now introduce some notation which we will use in the proof of
Theorem \ref{bigdomain}.

\smallskip

Let $K$ be a fixed constant, $\eps >0$, and let $\tau, |\tau| \leq
1$ be a vector lying in the hyperplane $\tau \cdot \textbf{e} =0$.
For any bounded function $w$ over a subdomain $\Omega \subset
\s_+^n$ we denote by $w^*=w^*(\tau,\eps)$ the corresponding
possibly multivalued function such that the surface
$X=e^{w+K\epsilon} z + \tau\epsilon$ can be represented as
$X=e^{w^{*}}z$ over its projection $\Omega_w^*$ on the unit upper
hemisphere $\s_+^n$. Precisely, let $e^{w(z)+K\eps} z+\tau
\eps=e^{w^*(z^*)} z^*$, with $z^*\in \Omega^*_w \subset  S^n_+$
and write $\rho=e^{w(z)},\rho^*=e^{w^*(z^*)}$. Then
\begin{align*}&z^*=\frac{z+\eps e^{-K\eps}\dfrac{\tau}{\rho}}{\sqrt{1+2\eps
e^{-K\eps} \dfrac{z\cdot
\tau}{\rho}+\dfrac{e^{-2K\eps}\eps^2|\tau|^2}{\rho^2}}},\\
& \ \\
&\rho^*=\sqrt{e^{2K\eps}\rho^2+2e^{K\eps}\eps \rho z \cdot \tau +
\eps^2|\tau|^2}.\end{align*}

\medskip

\noindent Note that if $w$ is Lipschitz with constant $L$, then
the mapping $z \rightarrow z^*$ is injective for $\eps \leq
\eps_0(L)$ and hence $w^*$ is well-defined and also Lipschitz.
Moreover, if $\overline{w}$ and $\underline{w}$ are both Lipschitz
with constant $L$ and $\overline{w}=\underline{w}=\varphi $ on
$\partial \Omega$ , then for $\eps \leq \eps_0(L)$,
$\Omega^*_{\overline{w}}=\Omega^*_{\underline{w}}$.

\smallskip

We are now ready to prove our Theorem.

\

\textit{Proof of Theorem $\ref{bigdomain}$}. Theorem
\ref{existmin} together with Proposition \ref{bar} guarantees the
existence (and uniqueness) of a minimizer $v$ to
$\mathcal{I}^\phi$ which is in the class $BV_M(\Omega) \cap
C(\partial \Omega).$ We need to show that $v \in C^{0,1}(\Omega)$.
Towards this aim we will prove the following claim.

\smallskip

 \noindent \textbf{Claim}: For any vector $\tau$, $|\tau| \leq 1$,
such that $\tau \cdot \textbf{e}=0$, and for all small $\epsilon
>0$, the hypersurface $X = e^{v(z)+K\epsilon}z + \epsilon \tau$ is
above the hypersurface $X=e^{v(z)}z$ in their common domain of
definition.

\smallskip

 \noindent Here $K$ denotes a big constant depending on
the Lipschitz constant of the barriers from Proposition \ref{bar}.


First we observe that the existence of barriers implies the
existence of two Lipschitz functions $\underline{v},\overline{v}$
such that $\underline{v} \leq v +K\epsilon \leq \overline{v}$
(here we are using Corollary \ref{maxprinc}), and $\underline{v}=
\overline{v}=\phi+K\epsilon$ on $\partial \Omega$.
Correspondingly, using the notation introduced before the proof,
$\underline{v}^*$ and $\overline{v}^*$ are Lipschitz functions for
small $\eps$, and
$\Omega^*_{\overline{v}}=\Omega^*_{\underline{v}}:=\Omega^*$.

We wish to prove that $v^*$ is a (single-valued) function over
$\Omega_v^*=\Omega^*$. Then the desired claim consist in showing
that $v^* \geq v$ in $\Omega \cap \Omega^*$, and it will follow
from the comparison principle Lemma \ref{uniqtrace}.

We use the notation at the end of Section 3. Let $C$ be the
radial cone over $ \Omega$, and set

$$\mathcal{V}+\epsilon\tau := \{E \subseteq C +
\epsilon\tau : E \ \text{measurable}, \ \underline{V} + \epsilon\tau \subseteq E \subseteq \overline{V}
+ \epsilon\tau\},$$

\medskip

\noindent where $A+ \epsilon \tau := \{x + \epsilon \tau, x\in
A\}$ for all $A \subset \R^{n+1}.$

\noindent Also, if $C^*$ is the radial cone over $
\Omega^*$, we let

$$\mathcal{V}^* = \{E \subseteq C^*:
E \ \text{measurable,} \ \underline{V}^* \subseteq E \subseteq
\overline{V}^*\},$$

\medskip

\noindent where $\underline{V}^*, \overline{V}^*$ denote
respectively the subgraphs in $C^*$ of $X=e^{\underline{v}^*}z$,
and $X=e^{\overline{v}^*}z$.

Notice that there is a one-to-one correspondence between
competitors in the classes $\mathcal{V}+\epsilon\tau ,
\mathcal{V}^*$ and the associated energies differ by a constant
(recall the definition of $w^*$).

Hence, since the subgraph of $X=e^{v+K\epsilon} z + \epsilon\tau$
minimizes $\mathcal{F}$ in $\mathcal{V} + \epsilon\tau$, then the
subgraph of $X=e^{v^*} z$ is a minimizer to $\mathcal{F}$ in
$\mathcal{V}^*$, and by the uniqueness result Proposition
\ref{minuniq} it is a graph over $\Omega^*$.

Now, in order to apply the comparison principle Lemma
\ref{uniqtrace}, we need to show that

\begin{enumerate}
\item $v^* \geq v$ on $\partial \Omega^* \cap \overline{\Omega};$
\item $v^* \geq v$ on $\partial \Omega \cap \overline{\Omega^*};$
\end{enumerate}

\smallskip

\noindent where the inequalities above are meant in the trace
sense (note that the existence of barriers implies that $v^*$ has
a continuous trace on $\partial \Omega^* \cap \overline{\Omega},$
while $v$ has a continuous trace on $\partial \Omega \cap
\overline{\Omega^*}$).

In order to prove $(1)$, we will show that $v^*$ is greater than
the upper barrier $\overline{v}$ for $v$ on $\partial \Omega^*
\cap \overline{\Omega}.$ Let $z \in
\partial \Omega^* \cap \overline{\Omega},$ and let $x \in \partial
\Omega$ be such that

\begin{equation*}\label{equality1}
e^{v^*(z)}z = e^{v(x)+ K\epsilon}x + \tau \epsilon.
\end{equation*}

\smallskip

\noindent It follows that

$$|e^{v^*(z)} - e^{v(x)+K\epsilon}| \leq \epsilon$$

\noindent and

$$|x-z| \leq C\epsilon,$$

\smallskip

\noindent with $C$ depending on the $L^\infty$ norm of $v$. If $K$
is very large, these two inequalities imply that

\begin{equation}\label{finalstep}
v^*(z) \geq v(x) + K^*|x-z|,
\end{equation}

\medskip

\noindent where $K^*$ is larger that the Lipschitz constant of the
upper barrier $\overline{v}$. Since $v(x)=\overline{v}(x)$
equation \eqref{finalstep} clearly gives $(1)$.

Part $(2)$ follows in the same way, using the lower barrier for
$v^*$. Thus our claim is proved.

\smallskip

We now show that our claim implies the Lipschitz continuity of
$v$.

Let $z \in \Omega$ and let $\mathcal{C}= \mathcal{C}(z,\theta)$ be
the circular cone with vertex at $e^{v(z)}z$, axis $z$, and
opening $\theta$. Since $\Omega$ is a strict subdomain of
$\s_+^n$, it is above the hyperplane $y=\delta$ (recall that
$y=z_{n+1}$), and thus each point $x$ can be represented as:

\begin{equation}\label{repr}x=e^{v(z)}z+\alpha z + \beta
\sigma,\end{equation}

\medskip

\noindent with $|\sigma|=1$, $\sigma \cdot \textbf{e}=0$, $\alpha,
\beta \geq 0$, and $\beta/\alpha \leq C(\theta,\delta)$ with
$C(\theta,\delta) \rightarrow 0$ as $\theta \rightarrow 0$.

Indeed, each point $x$ in the cone $\mathcal{C}$ can be
represented as

$$x=e^{v(z)}z + \gamma(z + \eta z_\perp)$$

\medskip

\noindent with $z_\perp$ unit vector in $T_z(\s_+^n)$, $\gamma
\geq 0,$ and $0 \leq \eta \leq \tan \theta \rightarrow 0$ as
$\theta \rightarrow 0.$

Now, let us decompose

$$z_\perp = az+b\sigma$$

\smallskip

\noindent with

$$a=\frac{z_\perp \cdot \textbf{e}}{z\cdot
\textbf{e}}; \quad b=\sqrt{1+a^2}.$$

\medskip

 \noindent Hence $\sigma \cdot \textbf{e}= 0,
|\sigma|=1.$ Moreover,

$$|a| \leq 1/\delta, \quad b \leq 2/\delta$$

\medskip

\noindent because $z$ is above the hyperplane $y=\delta.$

Therefore,

$$x=e^v(z)z + \gamma[(1+\eta a)z + b\eta\sigma]$$

\medskip

\noindent with the ratio

$$\frac{b\eta}{1+\eta a}$$

\medskip

\noindent going to zero as $\theta$ goes to zero.

Now, given $x$ (represented as in \eqref{repr}) in a neighborhood
$N$ (in $\mathcal{C}$) of $e^{v(z)}z$, that is for $\alpha$ small,
we can choose $\epsilon$ such that

$$e^{v(z)+K\epsilon} = e^{v(z)}+\alpha,$$

\medskip

\noindent hence $\epsilon=O(\alpha)$. Moreover, since
$\beta/\alpha \leq C(\theta,\delta) \rightarrow 0$ as $\theta
\rightarrow 0$, by choosing $\theta$ small enough depending on $K,
\|v\|_\infty, \epsilon_0(L), \delta$ we can guarantee that $\beta
\leq \epsilon.$ Hence

$$x = e^{v(z)+K\epsilon} + \epsilon \tau.$$

\medskip

\noindent Thus the set $S(\epsilon, \tau)= \{X=e^{v(z)+K\epsilon}z
+\epsilon \tau, 0\leq \epsilon \leq \epsilon_0(L), |\tau| \leq 1,
\tau \cdot \textbf{e}=0\}$ contains the cone $N \cap
\mathcal{C}(z,\theta)$.





Therefore, according to our claim at each point of the surface
$X=e^{v(z)}z$, there exists a small radial cone of fixed opening
which is completely above the surface. This geometric property
translates in the fact that for $q \in \Omega$ in a neighborhood
of $z$ we have

$$e^{v(q)} \leq e^{v(z)} + C(\theta) |z-q|.$$

\medskip

\noindent Since $v$ is bounded, this implies the Lipschitz
continuity of $v$. \qed

\

We state two simple corollaries of Theorem \ref{existence2}.

\begin{cor}\label{smalldomains} Let $B_\rho(P)$ be a ball in $\s^n_+ \cap  \{y \geq \epsilon\}$, for any
$\epsilon >0$, and let $\phi \in C^2(\s^n_+).$ Then there exists a
constant $r_0=r_0(n,H,\epsilon)$ such that the
Dirichlet problem \eqref{problem}-\eqref{bd} is uniquely solvable in $C^{\infty}(B_\rho(P))$, for all $\rho \leq
r_0.$
\end{cor}

\begin{cor}\label{caps} Let $S_\epsilon$ be the spherical cap $\s^n_+ \cap  \{y >\epsilon\}$, for any
$\epsilon >0$, and let $\phi \in C^2(\s^n_+).$ Then  the
Dirichlet problem \eqref{problem}-\eqref{bd} is uniquely solvable in $C^{\infty}(S_\epsilon).$
\end{cor}

\section{The interior gradient bound and the proof of Theorem \ref{anyd}}

\subsection{The interior gradient bound}In this subsection we prove the following interior gradient bound.

\begin{prop}\label{th2} Let $v$ be a
$C^{3}$ function satisfying equation \eqref{dirichlet} in
$B_{\rho}(P)\subset \{y \geq \epsilon\}$. Then \vspace{1mm}
$$W(P)\leq
C_1e^{\frac{C_2}{\rho^2}},$$

\vspace{1mm}

\noindent where $C_1,C_2$ are non-negative
constants depending only on $n,H,\epsilon$ and $
\|v\|_{L^{\infty}}$.
  \end{prop}

\begin{proof} Define the following linear elliptic
operator

\begin{equation}\mathcal{L} \equiv a^{ij}\nabla_{ij}-\frac2W a^{ij}W_i
\nabla_j-\frac{n}{y}\left(H \frac{\nabla  v}{W} +\bf{e}\right)\cdot \nabla\ee

\smallskip

\noindent where  $a^{ij}$ and $W$ are as in \eqref{local},\eqref{W}.

Throughout the proof, the constants may depend on $n,H,\epsilon$
and $\|v\|_{L^{\infty}}$. One can compute that

\begin{equation}\mathcal{L}W \geq -CW \quad \text{in} \ B_\rho(P),\ee

\smallskip

\noindent (for
details we refer the reader to Theorem 4.2 in \cite{GS}, formula
(4.16)).

We will derive a maximum principle for the function $h=\eta(x)W$
by computing $\mathcal{L} h$. Without loss of generality we may
assume $1 \leq v \leq C_0$. A simple computation gives

\be
\mathcal{L}h \geq W (M\eta-C\eta), \label{eq1} \ee

\smallskip

\noindent where

\be
\label{eq2} M \equiv a^{ij}\nabla_{ij} -\frac{n}y(H\frac{\nabla
v}W+\bf{e})\cdot \nabla. \ee

\

\noindent Note that $\displaystyle Mv={\frac{nH}{yW}}$. Choose

\be \eta (z) \equiv
g(\phi(z)); \;\; g(\phi) = e^{K \phi} - 1, \ee

\smallskip

\noindent with the constant
$K > 0$ to be determined and

\[ \phi (z) =
\left[ -\frac{v(z)}{2v(P)} + \left(1 -
\left(\frac{d_P(z)}{\rho}\right)^{2}\right) \right]^{+}.
\]

\

\noindent Here $d_P(z)$ is the distance function (on the sphere) from $P$, the
center of the geodesic ball $B_{\rho}(P)$ .

Since $v$ is positive, $\eta(z)$ has compact support in
$B_{\rho}(P)$. We will  choose $K$ so that $M\eta>C\eta $ on the
set where $h>0$ and $W$ is large (here $M$ is as in \eqref{eq2}).

A straightforward computation gives that on the set where $h>0$,

\begin{align*}
M\eta&=g'(\phi)\left(a^{ij}\nabla_{ij}\phi-\dfrac{n}y\left(H\dfrac{\nabla  v}W+\textbf{e}\right)\cdot \nabla \phi\right)
+ g''(\phi)a^{ij}\nabla_i  \phi \nabla_j \phi\\
& \ \\
&=Ke^{K\phi}\left\{-\dfrac1{2v(P)}\frac{nH}{yW} - \dfrac2{\rho^2}(d_Pa^{ij}\nabla_{ij}d_P+a^{ij}\nabla_i d_P \nabla_j d_P)\right.
\\& \ \\ & \quad \left.-\dfrac{n}{\rho^2 y}\left( H\frac{\nabla  v}W+\textbf{e}\right)\cdot d_P\nabla d_P\right\} \\
& \ \\
&
\quad +K^2 e^{K\phi}a^{ij}\left(\dfrac{v_i}{2v(P)}+\dfrac2{\rho^2}d_P\nabla_i d_P\right)\left(\dfrac{
v_j}{2v(P)} +\dfrac2{\rho^2}d_P\nabla_j d_P\right).\end{align*}

\

\noindent Using the definition of $a^{ij}$ we find ($\langle \cdot
, \cdot \rangle$ denotes the inner product with respect to the
induced Euclidean metric on $\Sigma$)

\begin{align*}\displaystyle
& a^{ij}\left(\dfrac{v_i}{2v(P)}+\dfrac2{\rho^2}d_P\nabla_i d_P\right)\left(\dfrac{
v_j}{2v(P)} +\dfrac2{\rho^2}d_P\nabla_j d_P\right)=\\
& \ \\
&\dfrac{|\nabla v|^2}{4v(P)^2}+\dfrac{2d_P}{\rho^2v(P)} \langle
\nabla v, \nabla d_P \rangle
+\dfrac{4d_P^2}{\rho^4}\left(1-\left(\langle \dfrac{\nabla v}W\
\nabla d_P\rangle \right)^2\right)~.
\end{align*}

\

\noindent Hence,

$$
 M\eta -C\eta\geq e^{K\phi}\left\{ K^2\left(\frac1{8{C_0}^2 }-\frac1{W^2}
\left(\frac1{\rho^2}+\frac1{8{C_0}^2}\right)\right)-CK\frac1{\rho^2}-C\right\}. $$

\

\noindent Therefore on the set where $h>0$ and $\displaystyle W >1+4\frac{C_0}{\rho}$
 we find

$$ M \eta - C\eta\geq e^{K\phi}\left\{\frac{K^2}{16C_0^2}
-CK\frac1{\rho^2}-C\right\}. $$

\

\noindent Thus, the choice $\displaystyle K=16CC_0\left(1+\frac{C_0}{\rho^2}\right)$  gives

$$M \eta
-C\eta \geq 15Ce^{K\phi}>0~$$

\

\noindent on the set where $h>0$ and $\displaystyle W
>1+4\frac{C_0}{\rho}$ . Hence by (\ref{eq1}) and the maximum
principle, $\displaystyle W \leq 1+4\frac{C_0}{\rho}$ at the point $Q$ where $h$
achieves its maximum. Therefore

\[h(P)=(e^{\frac{K}2}-1) W(P)\leq h(Q) \leq  (1+4\frac{C_0}{\rho})(e^K-1),
\]

\

\noindent and hence

\be W(P)\leq e^{\frac{CC_0^2}{\rho^2}} \ee

\

\noindent for a slightly larger constant $C$. This proves Proposition \ref{th2}.\\
\end{proof}

\subsection{Smoothness of minimizers in any dimension.} In this subsection we remove the dimensional constraint and prove the regularity result in Theorem \ref{anyd}. The proof follows the lines of the Euclidean case. We present it for the sake of completeness.

\

\textit{Proof of Theorem \ref{anyd}}. We use a standard
approximation argument. Let $B=B_\rho(P)$ be a ball in $\Omega,$
with $\rho \leq r_0$ and  $r_0$ as in Corollary \ref{smalldomains}.

Denote by $\tilde{S}:= \textrm{proj}_\Omega S.$ Since $\tilde{S}$ satisfies
$H_{n-6}(\tilde{S}) =0$, there exists a sequence $S_k$ of open sets, such
that

$$S_k\supset\supset S_{k+1}, \quad k=1,2,3... \quad \bigcap_{k\in \mathbb{N}}S_k =\tilde{S}$$
and also

$$H_{n-1}(S_k \cap \partial B) \rightarrow 0.$$

\

Now let $\phi_k$ be a smooth function on $\partial B$ satisfying

$$\phi_k = v \ \text{in} \ \partial B
\setminus S_k$$

\begin{equation}\label{uniformbound}\sup_{\partial B} |\phi_k| \leq 2 \sup_{\partial
B} |v|. \end{equation}

\

\noindent Let $v_k$ be the unique solution to the
Dirichlet problem with boundary data $\phi_k$ on $\partial B$ (see
Corollary \ref{smalldomains}). The functions $v_k$'s are smooth in $B$
and also according to \eqref{uniformbound} and Theorem \ref{existmin}

\begin{equation}\label{maxv}\sup_{B}|v_k| \leq M( \sup_{\partial
B} |v|).\end{equation}

\

\noindent We also have that the $v_k$ minimizes $\mathcal{I}_B(\cdot)$ among
all competitors with boundary data $\phi_k$ (see Remark
\ref{sol=min}). Hence,

\begin{equation}\label{finalmin}
\mathcal{I}_B(v_k) \leq \mathcal{I}^{\phi_k}_{B}(w)
\end{equation}

\

\noindent for every $w \in BV(B).$ In particular, for $w=0$,

\begin{equation}\label{w=0}
\mathcal{I}_B(v_k) \leq |B|+ \int_{\partial B} |\phi_k| dH_{n-1}
\leq C
\end{equation}

\

\noindent where in the last inequality we used \eqref{uniformbound}.

From \eqref{maxv} and the a priori estimate of the gradient
(Proposition \ref{th2}) we conclude that the gradients $\nabla
v_k$ are equibounded in every compact subset of $B.$ Hence, by
Ascoli-Arzela we can extract a subsequence, which we still denote
by $v_k$, which converges uniformly on compact subsets of $B$ to a
Lipschitz continuous function $\tilde{v}$. Moreover, by the lower
semicontinuity of $\mathcal{I}_B(\cdot)$ combined with
\eqref{maxv} and \eqref{w=0} we obtain

$$\int_{B} |\nabla
\tilde{v} | \leq C$$

\

\noindent and therefore $\tilde{v} \in W^{1,1}(B).$

We claim that $\tilde{v}$ has trace $v$ on $\partial B.$
Assuming that the claim is true, then passing to the limit in
\eqref{finalmin} with $w=v$ and remarking that $\phi_k \rightarrow
v$ in $L^1(\partial B)$ we have

$$\mathcal{I}_B(\tilde{v}) \leq \mathcal{I}_B(v).$$

\

\noindent Thus the function $\tilde{v}$ also minimizes $\mathcal{I}_B(\cdot)$
and by the uniqueness of minimizers (see Remark \ref{unique}) we obtain $v=\tilde{v}$ proving that
$v$ is Lipschitz continuous in $B.$ Hence, by elliptic regularity $v$ is analytic in $B$.

We are now left with the proof of the claim. Let $z_0 \in \partial B$  be a regular point for $v$. Then for $k$ large enough
 $z_0 \in \partial B  \setminus S_k$ and hence $\phi_j=v$ in a neighborhood of $z_0$ in $\partial B$, for all $j \geq k.$
 We can construct two $C^2$ functions $\underline{\phi}$ and $\ov{\phi}$ on $\partial B$, such that  $\underline{\phi}=\ov{\phi}= u$ in a neighborhood of $z_0$ and $\underline{\phi} \leq \phi_j \leq \ov{\phi} $ for all $j \geq k.$

 Now, we solve the Dirichlet problem with boundary data $\underline{\phi}, \ov{\phi}$ and denote the solutions respectively by
$\underline{v},\ov{v}$ (again we use Corollary \ref{smalldomains}). Then, $\underline{v} \leq v_j \leq \ov{v}$ for all $j \geq k$ and therefore $\underline{v} \leq \tilde{v} \leq \ov{v}$, which immediately yields $\tilde{v}(z_0)=v(z_0).$

Thus, $\tilde{v}=v$ at every regular point, which implies the desired claim since $H_{n-1}(\tilde{S})=0.$
  \qed

\medskip

We conclude this section by sketching the proof of Theorem \ref{global}.

\medskip

\textit{Proof of Theorem $\ref{global}$.} Assume that $\Gamma$ is represented by

$$X=e^\varphi z, z\in \partial
 \s_+^n,$$

\smallskip

\noindent with $\varphi \in C^2 (\overline{\s_+^n}).$
Then, according to Proposition \ref{bar} (see Remark \ref{strictbarrier})
we can find upper and lower barriers $\overline{v}$ and $\underline{v}$ coinciding with $\varphi$ on $\partial \s_+^n.$
For any small $\epsilon>0$, let $\psi_\epsilon$ be a smooth function on the spherical cap $S_\epsilon:=\s_+^n \cap \{y > \epsilon\}
$ such that
$\underline{v} \leq \psi_\epsilon \leq \overline{v}$ on the boundary of $\partial S_\epsilon.$ Let $v_\epsilon$
be a minimizer to
$\mathcal{I}^{\psi_\epsilon}_{S_\epsilon}(\cdot)$, which by our regularity
theory is smooth. By the comparison principle (Corollary \ref{maxprinc})
$\underline{v} \leq v_\epsilon \leq \overline{v}$ in $S_\epsilon$. By the interior a priori bound (Proposition
\ref{th2}) we can extract
a subsequence $v_{\epsilon_k}$ which converges uniformly on compacts of $\s^n_+$ to a function $v$ which solves
the equation and also $\underline{v} \leq v \leq \overline{v}$ in $\s_+^n.$ This implies the
continuity of $v$ up to the boundary.

Finally, if $\varphi$ is only continuous, we approximate it (from above and below) with $C^2$ functions, and conclude the argument by comparison with the barriers associated to the smooth approximated boundary data.

\qed


\begin{thebibliography}{100}

\bibitem{D} Dierkes, U.  {\it Singular minimal surfaces},  Geometric analysis and nonlinear partial differential equations, 177--193, Springer, Berlin 2003.
\bibitem{GT} Gilbarg D., Trudinger N., {\it Elliptic Partial differential equations of second order}, Springer-Verlag 1983.
\bibitem{G} Giusti E., {\it Minimal Surfaces and Functions of bounded Variation,} Monographs in Math. 80, Birkhauser 1984.
\bibitem{GS} Guan B., Spruck J., {\it Hypersurfaces of constant mean curvature in Hyperbolic space with prescribed asymptotic boundary at infinity,} Amer. J. Math. 122 (2000), 1039--1060.
\bibitem{N} Nitsche P-A., {\it Existence of prescribed mean curvature graphs in hyperbolic space,} Manuscripta math. 108 (2002), 349--367.
\bibitem{S} Simon L., {\it Lectures  on geometric measure theory}, Proceeding of the Centre for Mathematical Analysis, Australian National University 3, Canberra 1983.
\bibitem{Sp}Spruck, J., {\it Interior gradient estimates and existence theorems for constant mean curvature graphs in $M^n \times \R$}, Pure and Appl. Math. Quarterly 3 (2007), (Special issue in honor of Leon Simon, Part 1 of 2)1--16.




\end{thebibliography}
\end{document}